\documentclass[12pt]{article}
\title{The Hilbert Schmidt version of the commutator theorem for zero trace
  matrices
  } 

\author{Omer Angel\thanks{Supprted in part by NSERC, the Isaac Newton
    Institute and Simons Foundation.}
  \and Gideon Schechtman\thanks{Supported in part by the Israel Science
    Foundation.}
}

\date{March 2015}

\makeatletter
\renewcommand*{\@fnsymbol}[1]{\ensuremath{\ifcase#1\or {}\or \dagger\or \ddagger\else\@arabic{#1}\fi}}
\makeatother

\usepackage{amsfonts,amssymb,epsfig}
\usepackage{amsmath,amsthm}

\newtheorem{thm}{Theorem}
\newtheorem{lm}{Lemma}

\def\R{{\mathbb R}}
\def\Z{{\mathbb Z}}

\def\E{\mathbb{E}}

\newcommand{\Tr}{\mathrm{Tr}}

\begin{document}
\maketitle

\begin{abstract}
  Let $A$ be a $m\times m$ complex matrix with zero trace. Then there are
  $m\times m$ matrices $B$ and $C$ such that $A=[B,C]$ and $\|B\|\|C\|_2\le
  (\log m+O(1))^{1/2}\|A\|_2$ where $\|D\|$ is the norm of $D$ as an
  operator on $\ell_2^m$ and $\|D\|_2$ is the Hilbert--Schmidt norm of
  $D$. Moreover, the matrix $B$ can be taken to be normal. Conversely there
  is a zero trace $m\times m$ matrix $A$ such that whenever $A=[B,C]$,
  $\|B\|\|C\|_2\ge |\log m-O(1)|^{1/2}\|A\|_2$ for some absolute constant
  $c>0$.

\end{abstract}

\section{Introduction}
As is well known (or see e.g. [Fi]) a complex $m\times m$ matrix $A$ is a
commutator (i.e., there are matrices $B$ and $C$ of the same dimensions as
$A$ such that $A=[B,C]=BC-CB$) if and only if $A$ has zero trace. Let
$\|\E\|$ denote the operator norm of an $m\times m$ matrix (as a map
$E:\ell_2^m\to\ell_2^m$) and let $|\cdot|$ be any other norm on the space
of $m\times m$ matrices satisfying $|EF|\le \|E\||F|$ and $|FE|\le
\|E\||F|$ for all $m\times m$ matrices. In such a situation clearly if
$A=[B,C]$ then $|A|\le 2\|B\||C|$.

We are interested in the reverse inequality: If $A$ has zero trace are
there $m\times m$ matrices $B$ and $C$ such that $A=[B,C]$ and $\|B\||C|\le
K\|A\|$ for some absolute constant $K$? If not what is the behavior of the
best $K$ as a function on $m$?

In \cite{jos} this question was dealt with for $|\cdot|$ being the operator
norm $\|\cdot\|$. An upper bound on $K$ which is smaller than any power of
$m$ was given.

Here we deal with $|\cdot|$ being the Hilbert--Schmidt norm which we denote
$\|\cdot\|_2$.  We give matching upper and lower bounds (up to a constant
factor). 

\begin{thm}\label{thm:1}
  Let $A$ be an $m\times m$ matrix with zero trace, then there are $m\times
  m$ matrices $B$ and $C$ such that $A=[B,C]$ and $\|B\|\|C\|_2\le (c+\log
  m)^{1/2}\|A\|_2$. Moreover, the matrix $B$ can be taken to be
  normal. Conversely, for each $m$ there is a zero trace $m\times m$ matrix
  $A$ such that for any $m\times m$ matrices $B,C$ with $A=[B,C]$,
  $\|B\|\|C\|_2\ge \frac12(c'+\log m)^{1/2}\|A\|_2$, where $c,c'$ are some
  universal constants.
\end{thm}

The proof of the upper bound which is done by quite a simple random choice
is given in Section \ref{sec:upper}. The lower bound is a bit more involved
and is based on an idea from \cite{dfww} and in particular on a variation
on a lemma of Brown \cite{br} . The proof is given in Section
\ref{sec:lower}.

\section{The upper bound}\label{sec:upper}

Since both norms $\|\cdot\|$ and $\|\cdot\|_2$ are unitarily invariant and
since any zero trace matrix is unitarily equivalent to a matrix with zero
diagonal, we may and shall assume that $A$ has zero diagonal. In that case
we shall find a diagonal matrix $B=\Delta(b_1,b_2,\dots,b_m)$ with the
desired property. Note that translating back and assuming $A$ has merely
zero trace, the resulting $B$ is normal, being unitarily equivalent to a
diagonal matrix. 

Let $m>1$. If $A=[B,C]$ with $A$ with zero diagonal and
$B=\Delta(b_1,b_2,\dots,b_m)$ with all its diagonal entries distinct, then
necessarily $c_{i,j} = \frac{a_{i,j}}{b_i-b_j}$ for $i\not= j$.

Let $G\subset\Z+\Z\imath$ be the points with $m$ smallest absolute
values, so that $\max_{z\in G} \{|z|\} \leq 1+\sqrt{m/\pi}$.  Let
$\{b_i\}_{i=1}^m$ be a uniformly random permutation of these $m$ points, so
that necessarily $\|B\|\le 1 + \sqrt{m/\pi}$.  We now evaluate the
expectation of the resulting $\|C\|_2^2$.
\begin{equation}\label{eq:normofC}
\E\|C\|_2^2 = \E\sum_{i\not=j}\frac{|a_{i,j}|^2}{|b_i-b_j|^2} =
\sum_{i\not=j}|a_{i,j}|^2\E\frac{1}{|b_i-b_j|^2} = \|A\|_2^2
\E\frac{1}{|b_1-b_2|^2}.
\end{equation}
To evaluate $\E\frac{1}{|b_1-b_2|^2}$ fix $b_1\in G$. The expectation
conditioned on $b_1$ is
\begin{align*}
  \frac{1}{m-1} \sum_{\substack{b_2\in G\\b_2\neq b_1}} \frac{1}{|b_1-b_2|^2}
  &\le \frac{1}{m-1} \left[a_0 + \iint_{1\leq|z-b_1|\leq 2\sqrt{m/\pi}}
    \frac{|dz|^2}{|z-b_1|^2} \right] \\
  &\le \frac{a_1 + \pi\log m}{m}
\end{align*}
for some absolute constants $a_0,a_1$.

Plugging this into (\ref{eq:normofC}) we get that $\E\|C\|_2^2\le
\frac{a_1+\pi\log m}{m}\|A\|_2^2$ and thus there is a realization of the $b_i$-s
which gives $\|B\|\|C\|_2\le \sqrt{c+\log m} \|A\|_2$, for some absolute
constant $c$, as desired.

\paragraph{Remark.}
One can clearly replace the $m$ points of $G$ by another set of points in
the same disc about zero.  Sets minimizing such an energy function are a
well studied subject.  However, no significant improvement can be gained by
replacing $G$ with another set, and in particular our choice of $G$
achieves the optimal leading term $\pi m\log m$.  See for example \cite{hs}
in which tight bounds are given for a related quantity on the two
dimensional sphere.



\section{The lower bound}\label{sec:lower}

We begin with a Lemma which is a variation on a lemma of Brown \cite{br}

\begin{lm}\label{lm:brown}
  Assume $S,T$ are $m\times m$ matrices, $m\le \infty$, and $M$ is a finite
  dimensional subspace of $\ell_2^m$ (where $\ell_2^\infty=\ell_2$) such
  that for some $\lambda\in \mathbb{C}$\\ $([S,T]+\lambda
  I)(\ell_2^m)\subseteq M$. Then there are orthogonal subspaces
  $H_n\subseteq \ell_2^m$, $n=0,1,\dots$, with $H_0=M$,
  $\rm{dim}H_n\le(n+1)\rm{dim}M$, $n=1,2,\dots$, and $P_iSP_j=P_iTP_j=0$
  for all $i>j+1$, $j=0,1,\dots$. Here $P_l$ is the orthogonal projection
  onto $H_l$.  Moreover, $\sum_{n=0}^\infty\oplus H_n$ is invariant under
  $S$ and $T$.
\end{lm}

\begin{proof}
  Let $V_0=H_0=M$ and for $n\ge 1$ let $V_n$ be the linear span of
  $\{S^kT^lM; k+l\le n\}$. For $n\ge 1$ put $H_n=V_n\ominus
  V_{n-1}$. Clearly, $\rm{dim}H_n\le(n+1)\rm{dim}M$ and
  $\sum_{n=0}^\infty\oplus H_n$ is invariant under $S$ and $T$. To show
  that $P_iSP_j=P_iTP_j=0$ for all $i>j+1$ it is enough to show that
  $TV_n\subseteq V_{n+1}$ and $SV_n\subseteq V_{n+1}$ for all $n$.

The second containment is obvious. To prove the first it is enough to show
that for all $k\ge 1$ and $k+l\le n$, $TS^kT^l M \subseteq V_{n+1}$. Now,
\begin{align*}
  T S^k T^l &= S^k T^{l+1} + \sum_{i<k} S^i[T,S] S^{k-i-1} T^l \\
  &= S^k T^{l+1} - k\lambda S^{k-1} T^l + \sum_{i<k} S^i([T,S]+\lambda I)
  S^{k-i-1} T^l. 
\end{align*}
Now, the first term here has range in $V_{n+1}$ and the second in
$V_{n-1}\subseteq V_{n+1}$.  Since $[T,S]+\lambda I$ has range in $M$ the
$i$th term in the last sum has range in $S^i M \subseteq V_i \subseteq
V_{n+1}$, and the proof is complete. 
\end{proof}

Let $P$ be the rank one orthogonal projection onto the first coordinate in
$\ell_2^m$, $m<\infty$, given by the matrix
\[
P = \begin{pmatrix}
  1&0&\dots&0\\0&0&\dots&0\\\vdots&\vdots&\ddots&\vdots\\0&0&\dots&0
\end{pmatrix}
\]
and let $A=P-\frac{1}{m}I$. Obviously $A$ has zero trace and
Hilbert--Schmidt norm $\sqrt{1-\frac{1}{m}}$.  We now show that this $A$
gives the lower bound of Theorem~\ref{thm:1}.  Moreover, our argument gives
bounds on the leading singular values of $C$, based on the proof of Theorem
7.3 in \cite{dfww}, which also gives a lower bound on $\|C\|_2$.
Specifically, we get the following:

\begin{thm}\label{thm:lower}
Assume $A=[B,C]$ with $A$ as above, and the operator norm of $B$ equals
$1$.  Denote the singular values of $C$ as $s_1,s_2,\dots,s_m$, arranged in
non-increasing order.  Then for all $l\le m$,
\[
\sum_{i=1}^l s_i\ge \sqrt{l}/6.
\]
In particular the Hilbert--Schmidt norm of $C$ is at least $c\sqrt{\log m}$
for some absolute constant $c>0$.
\end{thm}

\begin{proof}[Proof of the lower bound in Theorem~\ref{thm:1}]
Let $M$ be the one dimensional subspace of $\ell_2^m$ spanned by the first
coordinate. Applying Lemma \ref{lm:brown} to this subspace with $S=B$,
$T=C$ and $\lambda=1/m$ we get orthogonal subspaces $M=H_0, H_1, \dots$
(which of course are eventually the zero subspace) with $\rm{dim}H_n\le
n+1$, so that $\sum_{n=0}^\infty \oplus H_n$ is invariant under $B$ and $C$
and $P_iBP_j=P_iCP_j=0$ for all $i>j+1>0$, where $P_l$ is the orthogonal
projection onto $H_l$.

Note that $P_0=P$.  Note also that $\sum_{n=0}^\infty \oplus H_n$ is
$\R^m$.  Indeed, a proper subspace of $\R^m$ containing $H_0$ which is
invariant under $B$ and $C$ is also invariant under $A$, and the
restriction of $A$ to such a subspace has zero trace which clearly can't
hold.

Now, on $H_0$, $A$ is just $(1-\frac1m)$, so 
\begin{align*}
\left(1-\frac{1}{m}\right)P_0 &= P_0[B,C]P_0\\
&= P_0BP_0P_0CP_0-P_0CP_0P_0BP_0\\
&\qquad +P_0BP_1P_1CP_0-P_0CP_1P_1BP_0.
\end{align*}
Similarly, for $k>0$, since $A_{|H_k}=\frac{-1}{m}I_{H_k}$, and using
$P_iBP_j=P_iCP_j=0$ for other $i,j$,
\begin{align*}
\frac{-1}{m}P_k &= P_k[B,C]P_k\\
&= P_kBP_{k-1}P_{k-1}CP_k-P_kCP_{k-1}P_{k-1}BP_k\\
&\qquad +P_kBP_kP_kCP_k-P_kCP_kP_kBP_k\\
&\qquad +P_kBP_{k+1}P_{k+1}CP_k-P_kCP_{k+1}P_{k+1}BP_k.
\end{align*}
Using the trace property (e.g., $\Tr(P_kBP_{k-1}P_{k-1}CP_k) =
\Tr(P_{k-1}CP_kP_kBP_{k-1}$)), we get that for all $n$,
\begin{align*}
1-\frac{1}{m}\sum_{k=0}^n{\rm rank} P_k
&= \sum_{k=0}^n \Tr(P_k[B,C]P_k)\\
&= \Tr(P_nBP_{n+1}CP_n)-\Tr(P_nCP_{n+1}BP_n).
\end{align*}
So, since $\|B\|=1$,
\begin{equation}\label{eq:trace}
1-\frac{1}{m}\sum_{k=0}^n{\rm rank}P_k
\le \|P_{n+1}CP_n\|_1 + \|P_nCP_{n+1}\|_1.
\end{equation}
Since ${\rm rank} P_k \leq k+1$, this gives a lower bound on the norms of
$P_nCP_{n+1}$ and $P_{n+1}CP_n$:
\begin{equation}
  \label{eq:normbd}
  \|P_{n+1}CP_n\|_1 + \|P_nCP_{n+1}\|_1 \geq 1 - \frac{1}{m} \binom{n+2}{2}.
\end{equation}
The matrices $P_{n+1}CP_n$ and $P_nCP_{n+1}$ have rank at most $n+1$, so
changing to other norms is not too costly, which allows us to bound from
below the Hilbert--Schmidt norm of $C$.

To complete the proof of the lower bound of Theorem~\ref{thm:1}, note that
for a matrix $M$ of rank $r$ we have $\|M\|_2^2 \geq \frac{1}{r}\|M\|_1^2$,
so
\begin{align*}
\|C\|_2^2 &\ge \sum_n \|P_n C P_{n+1}\|_2^2 + \|P_{n+1} C P_n\|_2^2 \\
&\ge \sum_n \frac{1}{n+1}
\left(\|P_n C P_{n+1}\|_1^2 + \|P_{n+1} C P_n\|_1^2\right) \\
&\ge \sum_n \frac{1}{2(n+1)}
\left(\|P_n C P_{n+1}\|_1 + \|P_{n+1} C P_n\|_1\right)^2 \\
&\ge \sum_n \frac{1}{2(n+1)} \left(1 - \frac{1}{m}\binom{n+2}{2}\right)^2.
\end{align*}
We take the sum over $n$ with $\binom{n+2}{2}<m$.  It is straightforward to
see that the last sum is $\frac14 \log m + O(1)$, giving the claimed lower
bound.
\end{proof}

\begin{proof}[Proof of Theorem~\ref{thm:lower}]
  Lemma 7.9 in \cite{dfww} (whose proof is simple, based on polar
  decomposition) says that there are partial isometries $V,W$ on $\ell_2^m$
  such that
\[
P_nVCP_n=|P_{n+1}CP_n| \ \  \mbox{and} \ \  P_nWC^*P_n=|P_{n+1}C^*P_n|.
\]
Consequently,
\[
P_n(VC+WC^*)P_n=|P_{n+1}CP_n| + |P_{n+1}C^*P_n|
\]
and by (\ref{eq:trace}),
\[
\Tr(P_n(VC+WC^*)P_n)\ge 1-\frac{1}{m}\sum_{k=0}^n{\rm rank}P_k.
\]
Fix a positive integer $k$ and let $E_k=\sum_{i=0}^kP_i$ and $r_k={\rm
  rank}E_k\le (k+1)(k+2)/2$. Denoting by $s_i(R)$ the singular values of
the operator $R$, we get that as long as $(k+1)(k+2)\le m$,
\begin{align*}
\sum_{i=1}^{(k+1)(k+2)/2}s_i(VC+WC^*)&\ge\sum_{i=1}^{r_k}s_i(E_k(VC+WC^*)E_k)\\
&\ge \sum_{n=0}^k \Tr(P_n(VC+WC^*)P_n)\ge \frac{k+1}{2}.
\end{align*}
Where we have used Weyl's inequality to deduce the second inequality. It
follows that for all $k$ as above $\sum_{i=1}^{(k+1)(k+2)/2}s_i(C)\ge
\frac{k+1}{4}$. The main assertion of the theorem follows easily from that.

As for the last assertion, it is well known that it follows from the
first. Indeed, the non-increasing sequence $s_1,s_2,\dots, s_m$ majorizes a
sequence equivalent (with universal constants) to $1,2/\sqrt2,1/\sqrt3,
\dots, 1/\sqrt{m}$. Consequently,
\[
(\sum_{i=1}^m s_i^2)^{1/2}\ge c(\sum_{i=1}^{m} 1/i)^{1/2}\ge c^\prime (\log m)^{1/2}.
\]
\end{proof}

\paragraph{Acknowledgement.} We benefitted a lot from discussions with Bill
Johnson concerning the material of this note. In particular, he is the one
who pointed \cite{dfww} to us.

%
%

\begin{tabular}{ll}
O. Angel                       &G. Schechtman\\
Department of Mathematics      &Department of Mathematics\\
University of British Columbia &Weizmann Institute of Science\\
Vancouver, BC, V6T 1Z2, Canada &Rehovot 76100, Israel\\
{\tt angel@math.ubc.ca}        &{\tt gideon@weizmann.ac.il}
\end{tabular}

\end{document}